\documentclass[12pt,reqno]{amsart}
\usepackage{graphicx}
\newtheorem{thm}{Theorem}[section]

\theoremstyle{definition}

\theoremstyle{remark}

\numberwithin{equation}{section}

\begin{document}

\title[Burnside's formula]{An alternating proof of sharp inequalities related with Burnside's formula}%
\author{necdet batir}%
\address{department of  mathematics, nev{\c{s}}ehir hbv university, nev{\c{s}}ehir, 50300 turkey}%
\email{nbatir@hotmail.com}%


\begin{abstract} We provide an alternating proof of the following sharp inequalities  related with Burnside's formula for $n!$ proved by the author in [A double inequality related with Burnside’s, formula,  Proyecciones J.  Math., Vol. 37, No 1, pp. 55-59, 2018.]
\begin{equation*}
\sqrt{2\pi}\left(\frac{n+a_*}{e}\right)^{n+a_*}<n!<\sqrt{2\pi}\left(\frac{n+a^*}{e}\right)^{n+a^*}\,(n\in\mathbb{N}),
\end{equation*}
where the constants $a_*=0.428844044...$ and $a^*=0.5$ are the best possible.
\end{abstract}
\maketitle
\section{introduction}
It is of interest to study how the factorial function $n!$ behaves when $n$ is sufficiently large. The most well known approximation formula to approximate factorial function is the Stirling's formula given by
\begin{equation}\label{e:1}
n!\sim n^ne^{-n}\sqrt{2\pi n}.
\end{equation}
The origin of this formula is based on a study of the French mathematician Abraham de Moivre. In 1733 de Moivre developped the formula
\begin{equation*}
n!\sim C\cdot\sqrt{n}n^ne^{-n},
\end{equation*}
where $C$ is a constant. He was unable, however, to evaluate exact numerical value of this constant; this task befell a Scot mathematician James Stirling (1692-1770), who found  $C=\sqrt{2\pi}$. Formula (\ref{e:1})
is  known as Stirling's formula today. This formula has many applications in probability theory, statistical physics. number theory, and special functions. The most important  known simple  approximation formula for factorial function after Stirling's formula is Burnside's formula \cite{3}, which is given by
\begin{equation*}
n!\sim\sqrt{2\pi}\left(\frac{n+\frac{1}{2}}{e}\right)^{n+\frac{1}{2}}.
\end{equation*}
It is well known that this formula is more accurate than the Stirling's formula.
Please refer to \cite{4,5,6,7} and the references therein for related inequalities.
In very new paper the author \cite{2} proved that
\begin{equation*}
\sqrt{2\pi}\left(\frac{n+a_*}{e}\right)^{n+a_*}<n!<\sqrt{2\pi}\left(\frac{n+a^*}{e}\right)^{n+a^*}\,(n\in\mathbb{N}),
\end{equation*}
where the constants $a_*=0.428844044...$ and $a^*=0.5$ are the best possible. Our aim in this note is to supply a completely different proof of these inequalities.

\section{main result}
Our main result is a new proof of the following theorem.
\begin{thm} For all $n\in\mathbb{N}$ we have
\begin{equation}\label{e:2}
\sqrt{2\pi}\left(\frac{n+a_*}{e}\right)^{n+a_*}<n!<\sqrt{2\pi}\left(\frac{n+a^*}{e}\right)^{n+a^*},
\end{equation}
where the constants $a_*=0.428844044...$ and $a^*=0.5$ are the best possible.
\end{thm}
\begin{proof}
Let $n$ be a fixed positive integer and $x$ be a non-negative real number. We define
\begin{equation*}
  f(x):=\sqrt{2\pi}\left(\frac{x+n}{e}\right)^{x+n}.
\end{equation*}
Differentiation gives $f'(x)=f(x)\log (x+n)>0$. So, $f$ is strictly increasing for all $x>0$. On the other hand, we have
\begin{equation*}
f(0)= n^ne^{-n}\sqrt{2\pi}<n!
\end{equation*}
and
\begin{equation*}
f(1/2)=\sqrt{2\pi}\left(\frac{n+\frac{1}{2}}{e}\right)^{n+\frac{1}{2}}>n!,
\end{equation*}
see \cite{1}, that is,
\begin{equation*}
  f(0)<n!<f(1/2).
\end{equation*}
Since $f$ is continuous on $(0,1/2)$, the intermediate value theorem yields that there exist a sequence $(a_n)$ with $a_n\subset(0,1/2)$ such that
\begin{equation}\label{e:3}
  f(a_n)=n!.
\end{equation}
Since $f$ is continuous and monotonic increasing on $(0,1/2)$, it has an inverse $f^{-1}$ and $f^{-1}$ is also monotonic increasing on the interval
$$
( \sqrt{2\pi}e^{-1},\infty).
$$
We therefore get from (\ref{e:3}) that
$$
a_{n+1}=f^{-1}((n+1)!)>f^{-1}(n!)=a_n,
$$
which implies that  $(a_n)$ is strictly increasing. Since $(a_n)$ is bounded and strictly increasing, it has a limit and we have for all $n\in\mathbb{N}$
\begin{equation}\label{e:4}
a_1<a_n<\lim\limits_{n\to\infty}a_n.
\end{equation}
If we take $n=1$ in  (\ref{e:3}) we obtain
\begin{equation*}
  1=\sqrt{2\pi}\left(\frac{a_1+1}{e}\right)^{a_1+1}
\end{equation*}
or taking the logarithm of both side
\begin{equation*}
  a_1+1-(a_1+1)\log(a_1+1)-\log\sqrt{2\pi}=0.
\end{equation*}
By the help of computer program \textit{Mathematica}, we can solve this equation and we find that $a_1=0.428844044...$. Now we want to show that $\lim\limits_{n\to\infty}a_n=1/2$. It is difficult to evaluate it in the usual way. Instead, we shall follow a different way. From
\begin{equation*}
  n!=\sqrt{2\pi}\left(\frac{n+a_n}{e}\right)^{n+a_n},
\end{equation*}
we get
\begin{equation}
\lim\limits_{n\to\infty}\frac{\sqrt{2\pi}\left(\frac{n+a_n}{e}\right)^{n+a_n}}{n!}=1.
\end{equation}
Using Stirling's formula this becomes
\begin{equation*}
\lim\limits_{n\to\infty}\frac{(n+a_n)^{a_n}e^{-a_n}\left(1+\frac{a_n}{n}\right)^n}{\sqrt{n}}=1.
\end{equation*}
Since $a_n$ is convergent it is easy to see that $\lim\limits_{n\to\infty}e^{-a_n}\left(1+\frac{a_n}{n}\right)^n=1$, so that
\begin{equation*}
\lim\limits_{n\to\infty}\frac{(n+a_n)^{a_n}}{\sqrt{n}}=\lim\limits_{n\to\infty}\left(1+\frac{a_n}{n}\right)^{a_n}\lim\limits_{n\to\infty}n^{a_n-\frac{1}{2}}=1.
\end{equation*}
The first limit on the right hand side goes to 1 as $n$ approaches to infinity, so we have $\lim\limits_{n\to\infty}n^{a_n-\frac{1}{2}}=1$, which is possible only if $\lim\limits_{n\to\infty}a_n=\frac{1}{2}$.
Let us set $a_*=a_1=0.428844044...$ and $a^*=0.5$. Then by (\ref{e:4}) we have $a_*<a_n<a^*$ for all $n\in\mathbb{N}$. Applying $f$ to each side leads to
$$
f(a_*)<f(a_n)=n!<f(a^*),
$$
which is equivalent to (\ref{e:2})with the possible  constants $a_*=0.428844044...$ and $a^*=0.5$. This completes the proof.
\end{proof}

\bibliographystyle{amsplain}

\end{document}